\documentclass[11pt, oneside,reqno]{article}
\usepackage{graphicx}
\usepackage{amsfonts }
\usepackage{amsmath}
\usepackage{mathabx}

\usepackage{fullpage}
\usepackage{amssymb}
\usepackage{amsthm}
\usepackage{tikz}
\usepackage{mathtools}
\usepackage{lmodern}
\usepackage{paralist}
\usepackage[nodisplayskipstretch]{setspace}
\newtheorem{thm}{Theorem}[section]
\newtheorem{conj}[thm]{Conjecture}
\newtheorem{lemma}[thm]{Lemma}

\newtheorem{cor}[thm]{Corollary}

\theoremstyle{definition}
\newtheorem{rem}[thm]{Remark}

\newcommand{\B}{\overline{B_1(0)}}
\newcommand{\R}{\mathbb{R}}

\newcommand{\St}{\text{st}}
\newcommand{\lk}{\text{link}}

\newcommand{\Sp}{\mathbb{S}}

\newcommand{\dK}{\partial K}
\newcommand{\dP}{\partial P}
\newcommand{\dB}{\partial B}

\newcommand{\dH}{\delta^H}

\newcommand*\longtwoheadrightarrow{\ensuremath{\relbar\joinrel\twoheadrightarrow
}}

\begin{document}
\title{A Geometric Lower Bound Theorem}
\author{Karim Adiprasito\thanks{K.~A.~Adiprasito acknowledges support by a Minerva postdoctoral fellowship of the Max Planck Society, and NSF Grant DMS 1128155.
}\\
\small Einstein Institute for Mathematics \\ [-0.8 ex]
\small Hebrew University of Jerusalem \\ [-0.8 ex]
\small Jerusalem, 91904 Israel \\ [-0.8 ex]
\small \texttt{adiprasito@math.huji.ac.il} \\[-0.8 ex]
\and Eran Nevo\thanks{Research of E.~Nevo was partially supported by Israel Science Foundation grants ISF-805/11 and ISF-1695/15.
} \\
\small Einstein Institute for Mathematics \\ [-0.8 ex]
\small Hebrew University of Jerusalem \\ [-0.8 ex]
\small Jerusalem, 91904 Israel \\ [-0.8 ex]
\small \texttt{nevo@math.huji.il} \\
\and Jos\'e Alejandro Samper\thanks{J.~A.~Samper thanks Isabella Novik for the research assistant positions funded through NSF Grants DMS-
1069298 and DMS-1361423.}\\
\small Department of Mathematics\\[-0.8ex]
\small University of Washington\\[-0.8ex]
\small Seattle, WA 98195-4350, USA\\[-0.8ex]
\small \texttt{samper@math.washington.edu}
}
\date{\today}

\maketitle
\begin{center}
\emph{Dedicated to Gil Kalai on the occasion of his 60th birthday.}
\end{center}
\begin{abstract}
We resolve a conjecture of Kalai relating approximation theory of convex bodies by simplicial polytopes to the face numbers and primitive Betti numbers of these polytopes and their toric varieties. The proof uses higher notions of chordality. Further, for $C^2$-convex bodies, asymptotically tight lower bounds on the $g$-numbers of the approximating polytopes are given, in terms of their Hausdorff distance from the convex body.
\end{abstract}
\section{Introduction}
The combinatorial structure of polytopes
was studied since antiquity and has been one of the major topics in algebraic and geometric combinatorics in the last few decades. The simplest combinatorial invariant of a $d$-polytope $P$ is the $f$-vector $(f_{-1}, f_0, \dots, f_{d-1})$, where $f_i$ is the number of $i$-dimensional faces of $P$. Understanding face numbers of polytopes is one of the oldest branches of mathematics.

The celebrated $g$-theorem, conjectured by McMullen \cite{mcmullen1970},  gives a complete characterization of the $f$-vectors of \emph{simplicial} polytopes, namely polytopes all whose proper faces are simplices.
It is conveniently phrased is terms of the $g$-vector, obtained by a linear transformation of the $f$-vector.
Billera and Lee \cite{Billera-Lee} proved sufficiency of the numerical conditions and Stanley \cite{Stanley:NumberFacesSimplicialPolytope-80}
proved their necessity by relating the $g$-numbers to the primitive Betti numbers of the associated projective toric varieties. Some extremal cases in terms of the $g$-numbers are well understood; for instance polytopes with $g_k=0$ are exactly the $(k-1)$-stacked polytopes, as stated in the Generalized Lower Bound Conjecture (GLBC) of McMullen-Walkup \cite{McMullenWalkup:GLBC-71} and recently proved by Murai-Nevo \cite{Murai-Nevo:GLBC}. However, away from the extremal primitive Betti vectors, the simplicial polytopes become much harder to understand.

An equally foundational subject in polytope theory is approximation theory. Polytopes are dense in the space of convex bodies with respect to several different metrics, and the question what is the minimal number of faces of a certain dimension that are needed to produce an approximation of a certain quality has been substantially studied; see Schneider \cite{Schneider}, Gruber \cite{MR949830, MR1116357}, and finally B\"or\"oczky \cite{MR1770932, MR1742207}, producing asymptotically tight answer for the \emph{individual} face numbers for $C^2$-convex bodies.

In 1994 Kalai \cite{Kalai:Aspects-94} posed a visionary conjecture that relates the entire $f$-vector of a simplicial polytope $P$ to its metric structure.
Roughly speaking, Kalai conjectures that if $K$ is a convex body whose boundary is of type $C^1$ and $P$ is a simplicial polytope that is close to $K$ in the Hausdorff distance, then the $f$-vector of $P$ must be far away from extremal $f$-vectors in the sense of the $g$-theorem. Kalai states his conjecture using the $g$-vector and shadow functions $\partial^k$ (see \cite[Section 8.5]{MR1311028}):



\begin{conj}[Kalai \cite{Kalai:Aspects-94}]\label{Kalai}
Let $K$ be a $C^1$-convex body in $\R^d$ and let $\{P_n\}_{n=1}^{\infty}$ be a sequence of simplicial polytopes that converges to $K$ in the Hausdorff metric. Then
\begin{itemize}
\item[(i)] for every $1\le k \le \left\lfloor \frac d2 \right\rfloor$,
\[\lim_{n\to \infty} g_k(P_n) = \infty, \]
\item[(ii)] and for every $1\le k \le \left\lfloor \frac d2 \right\rfloor -1$, 
\[ \lim_{n\to \infty} \left(g_k -\partial^{k+1}g_{k+1} \right)= \infty. \]
\end{itemize}
\end{conj}

The aim of this paper is to resolve part (i) of Conjecture~\ref{Kalai} and provide a quantitative lower bound on the $g$-numbers in the case when the boundary of $K$ is of type $C^2$. This provides the first bridge between the approximation theory by convex polytopes and the Stanley-Reisner theory of convex polytopes. From the geometric point of view,
it connects the geometry of the toric variety of the approximating polytope with the geometry of the underlying polytope.
More specifically,
this result shows that there is an intimate relation between the metric structure of some embeddings of a polytope and the primitive Betti numbers in the cohomology ring of the associated toric variety. On the other hand, our quantitative results generalize the theorems of B\"or\"oczky in the case when the approximating polytopes are simplicial.

Although in this paper we focus mainly on the Hausdorff metric, most of the results hold for other metrics, such as Schneider's metric, the Banach--Mazur distance, the symmetric difference distance, etc.\ as we rely on B\"or\"oczky's method \cite{MR1770932} for the final approximation.

In \cite{ANS15} we provided a notion of higher chordality of simplicial complexes and showed that it generalizes the classical notion of chordal graphs. In \cite{AdipII} the first named author introduced toric chordality, a powerful algebraic tool to study chordality in the stress-space of the simplicial complex as studied by Lee \cite{MR1384883}. He related this algebraic notion of chordality to the higher chordality notions of \cite{ANS15} and derived, among many other results,
a quantitative version of the GLBC in terms of the topological Betti numbers of induced subcomplexes. In this paper, we use this result to prove Kalai's lower bound conjecture in full generality (alternatively, for self-containedness, we use a weaker statement proved in the Appendix).

This paper is organized as follows:
in Section~\ref{sec:pre} we provide the needed preliminaries, in Section~\ref{sec:g_2} we give a simple proof of Conjecture~\ref{Kalai}(i) for the unit $4$-ball, using framework rigidity arguments. These arguments are vastly generalized in Section~\ref{sec:C^1} to prove Conjecture~\ref{Kalai}(i) in full generality, for $C^1$-convex bodies.
In Section~\ref{sec:C^2} we generalize B\"or\"oczky's results by
giving asymptotically tight lower bounds on the $g$-numbers when approximating a $C^2$-convex body, in terms of its Hausdorff distance from the approximating simplicial polytope. We also observe that Conjecture~\ref{Kalai}(ii) holds for approximations by random polytopes.

\section{Preliminaries}
\label{sec:pre}
\subsection{Convex bodies}
A convex body $K$ in $\mathbb{R}^d$ is a convex compact subset of $\mathbb{R}^d$ with non-empty interior. The main example of a convex body is the closed unit ball $\B$ in $\mathbb{R}^d$ with the standard metric. In general, every convex body is a convex embedding of $\B$ in $\R^d$. The boundary of a convex body $K$ is denoted by $\partial K$,
and $\mathbb{S}^{d-1}:=\partial \overline{B_1(0)}$ denotes the standard unit sphere.

Endow $\R^d$ with the standard inner product denoted by $\langle\,, \, \rangle$. For an element $u \in \Sp^{d-1} \subseteq \R^d$ and a convex body $K$, let $c(u, K) = \max_{\{v\in K\}} \langle u,v \rangle$. Also, let $H^+(u, K) = \{s \in \R^d, \langle u, s\rangle \le c(u,K)\}$ be a supporting halfspace of $K$ in direction $u$. It is well known that $K=\bigcap_{u\in \Sp^{d-1}} H^+(u,K)$. The boundary of $H^+(u,K)$ is denoted by $H(u,K)$. For a point $x\in \partial K$ there is at least one point $u\in \Sp^{d-1}$ such that $x\in H(u, K)$. If this point $u$ is unique we say that $x$ is non-singular. Denote the unique such direction by $u(x)$, whenever $x$ is a non singular point. For every non singular point $x$ there exist neighborhoods $U_x\subseteq \partial K$ and $V_x \subseteq H(u(x), K)$ of $x$, where $V_x$ is convex, and a non-negative convex function $f_x: V_x \to \R$, such that, for every $v$ in $V_x$, the point $\varphi_x(v) = v - f_x(v)u(x)$ is an element of $U_x$ and the map $\varphi_x$ is a homeomorphism from $V_x$ to $U_x$.

Endowing $\B$ with its standard differential structure, we say that a convex body $K$ is of \emph{type $C^k$} if it is the image of a $C^k$-embedding of $\B$ in $\R^d$. Equivalently, the boundary $\dK$ is a $C^k$-hypersurface in $\R^d$. If $k \ge 1$, and $K$ is a $C^k$-convex body, then every point $x\in \partial K$ is non-singular.

\subsection{Polytopes and simplicial complexes}
A \emph{polytope} $P$ is the convex hull of finitely many points in some Euclidean space; equivalently it is a bounded intersection of finitely many closed half-spaces. Polytopes are a very special class of convex bodies. A \emph{face} of a polytope $P$ is the intersection of a supporting hyperplane of $P$ with $P$. The \emph{dimension} of a face is the dimension of its affine span. Assume that $P$ is $d$-dimensional. The \emph{$f$-vector}  of $P$ is the vector $f_P:= (f_{-1},\, f_{0},\, f_{1}, \dots, \,f_{d-1})$ where $f_i$ is the number of $i$-dimensional faces of $P$ ($f_{-1}=1$ for the empty face). A \emph{simplex} is the convex hull of a set of affinely independent points, thus a $k$-dimensional simplex has $k+1$ vertices. A polytope $P$ is \emph{simplicial} if all proper faces of $P$ are simplices. We denote the set of proper faces of $P$ by $\partial P$ and call it the \emph{boundary} of~$P$.

A (geometric) \emph{simplicial complex} $\Delta$ is a finite family of simplices such that (i) if $F$ is in $\Delta$ and $G$ is a face of $F$, then $G$ is also in $\Delta$,  and (ii) for any two elements $F$ and $G$ of $\Delta$, $F\cap G$ is a face of both $F$ and $G$. Note that a polytope $P$ is simplicial if and only if the boundary of $P$ is a simplicial complex. The elements of a simplicial complex are also called \emph{faces} and the \emph{dimension} of a simplicial complex is the maximal dimension of a face. As in the case of polytopes we may define the $f$-vector of $\Delta$, $f_{\Delta}:=(f_{-1}, f_0, \dots, f_{d-1})$,  to be the vector such that $f_i$ is the number of faces of dimension $i$, called \emph{$i$-faces}.
Thus, for $\Delta=\partial P$, $f_{\Delta}=f_P$.

The set of faces of $\Delta$ of dimension at most $i$ is a subcomplex called the \emph{$i$-th skeleton} of $\Delta$ and denoted by $\Delta^{(i)}$. The set of $0$-faces is denoted by $V(\Delta)$ and is called the set of \emph{vertices} of $\Delta$; the $1$-faces are called \emph{edges}. When all faces of $\Delta$ that are maximal under inclusion have the same dimension $d$ we say $\Delta$ is \emph{pure} and refer to its $d$-faces as \emph{facets} and to its $(d-1)$-faces as \emph{ridges}.

The \emph{link} of a face $F$ of $\Delta$, denoted by $\lk_\Delta(F)$, or $\lk(F)$ for short, is the set of all faces $G$ of $\Delta$, such that $F\cap G = \emptyset$ and  $G$ is contained in a face that contains $F$. It is straightforward (see \cite[Proposition 2.4, page 55]{MR1311028}) that for every face $F$ of a simplicial polytope $P$  the link of $F$ in $\partial P$ is combinatorially isomorphic to the boundary of some simplicial polytope. The link of a vertex is sometimes called a \emph{vertex figure}. For a subset $W$ of the vertex set of $\Delta$, let $\Delta_W$ denote the \emph{induced subcomplex} of $\Delta$ on $W$, namely the complex whose faces are the subsets of $W$ which are faces of $\Delta$.

For a simplicial complex $\Delta$, let $\tilde{H}_k(\Delta)$ be the reduced $k$-th (simplicial or singular) homology group over $\mathbb{Q}$ and let $\tilde{\beta}_k(\Delta) := \dim_{\mathbb{Q}} \tilde{H}_k(\Delta)$ be the $k$-th topological Betti number. We say that a cycle (either simplicial or singular) is not trivial if its homology class does not vanish. Simplicial cycles can be viewed as singular cycles.

For a simplex $\Gamma$ in $\R^d$ of dimension $<d$, and $v$ a point not in the affine span of $\Gamma$, let $v*\Gamma = \text{conv}(v, \Gamma)$. The simplex $v*\Gamma$ is called the cone over $\Gamma$ with apex $v$.

A point set in $\mathbb{R}^d$ is \emph{generic}, or in \emph{general position}, if any $d+1$ of its points are affinely independent. An affine subspace is \emph{generic} w.r.t. a collection of geometric simplices if it contains no vertex, and its parallels contain no edge, of these simplices.

\subsection{f-vectors of simplicial polytopes}
The \emph{$f$-polynomial} of a $d$-dimensional simplicial polytope $P$ is the generating function of the $f$-vector, given by the polynomial $f_P(x) = \sum_{j=0}^d f_{j-1}x^j$. Sometimes it is convenient to consider the \emph{$h$-polynomial}, $h_P(x) := (1-x)^df_P\left(\frac x{1-x}\right)$. The $h$-vector $(h_0, h_1, \dots, h_d)$ of $P$ is the vector of coefficients of the $h$-polynomial, that is, $h_P(x) = \sum_{i=0}^d h_ix^i$. Knowing the $h$-vector is equivalent to knowing the $f$-vector. The Dehn-Sommerville relations (see \cite[Theorem 3.2 and Prop. 3.3]{MR0189039}) assert that $h_i = h_{d-i}$ for a
simplicial $d$-polytope $P$ and $0\le i \le d$. It follows that the first half of the entries of the $f$-vector of $P$ determine the entire $f$-vector of $P$.

The celebrated classification by  \cite{Billera-Lee} and \cite{Stanley:NumberFacesSimplicialPolytope-80} of the $f$-vectors of simplicial $d$-polytopes is known as the $g$-theorem and is usually stated in terms of the $g$-vector $(g_0, g_1, \dots, g_{\lfloor \frac d2 \rfloor})$, where $g_0:= h_0=1$ and $g_i = h_i - h_{i-1}$ for $1\le i \le \lfloor \frac d2 \rfloor$. To prove Conjecture~\ref{Kalai}(i) we only require the lower bound part of this theorem that states the nonnegativity of the $g_i$.

\begin{thm}{\label{gt}} ($g$-theorem) An integer vector $(g_0, g_1, \dots, g_{\lfloor \frac d2 \rfloor})$ is the $g$-vector of a simplicial $d$-polytope if and only if it is the Hilbert function of some graded commutative algebra finitely generated in degree 1. In particular, $g_0 = 1$ and $g_k \ge 0$ for $1\le k \le \lfloor \frac d2 \rfloor$.
\end{thm}
A numerical characterization of the Hilbert functions as in the $g$-theorem is due to Macaulay, using his shadow functions $\partial^k(\cdot)$, cf. \cite[Section 8.5]{MR1311028}. We will use them only in our last remark, Remark~\ref{rem:barany}, on Conjecture~\ref{Kalai}(ii).



The following recent result of Adiprasito \cite{AdipII} generalizes the lower bound theorem, and will be crucial in our proof of Conjecture~\ref{Kalai}(i).

\begin{thm}[The quantitative lower bound theorem]\label{thm:QLBT}
Let  $P$ be a simplicial $d$-polytope with boundary complex $\Delta$, $k\le \frac d2$,  and let $W$ be any subset of the vertices, then:
\begin{equation}\label{eqn:QLBT}\tilde{\beta}_{d-k-1}(\Delta_W) \le g_k(\Delta). \end{equation}
\end{thm}

The proof uses a subtle approach via combinatorial Morse theory. We will therefore, for purposes of self-containedness, provide also a slightly weaker alternative lemma to the same effect based on the McMullen proof of the hard Lefschetz theorem, see Lemma~\ref{lem:QLBT}.



\subsection{Framework rigidity}
Let $G = (V,E)$ be a graph and let $\varphi: V \to \R^d$ be any map. We say that $\varphi$ is \emph{rigid} if there exists $\varepsilon > 0$ such that if $\varphi': V \to \R^d$ is such that $|\varphi(v) - \varphi'(v)| < \varepsilon$ for any $v\in V$
and $|\varphi(v) - \varphi(w)| = |\varphi'(v)-\varphi'(w)|$ for every $\{w,v\} \in E$, then $|\varphi(v) - \varphi(w)| = |\varphi'(v)-\varphi'(w)|$ for every $\{w,v\} \in \binom V2$. We say that $\varphi$ is \emph{flexible} if it is not rigid.

The set of all maps $V\to \R^d$ forms a $d\cdot|V|$-dimensional vector space over $\R$ which can be endowed with the Lebesgue measure. A graph $G$ is \emph{generically $d$-rigid} if almost every map $\varphi: V\to \R^d$ is rigid and \emph{generically $d$-flexible} if almost every such map is flexible. It is known that every graph is either generically $d$-rigid or generically $d$-flexible.

Fix a vertex set $V$ and consider the family ${\cal R}(V,d) \subseteq 2^{\binom V2}$ of all the minimal under inclusion edge sets $E$ such that $G=(V,E)$ is a generically $d$-rigid graph. The collection ${\cal R}(V,d)$ is the set of bases of a matroid. In particular, the cardinality of any element of ${\cal R}(V,d)$ is an invariant denoted by $\rho(V,d)$.

Let $G= (V,E)$ be graph and let $\varphi:V\rightarrow \mathbb{R}^d$ be a map. A \emph{stress}, w.r.t. $(G,\varphi)$, is a map $\omega: E \to \mathbb{R}$ such that for every vertex $v$:
\begin{equation}\label{eqn:stress} \sum_{u:\ \{u,v\}\in E} \omega(\{u,v\})(\varphi(u)-\varphi(v)) = 0. \end{equation}
The family of stresses  of $(G, \varphi)$ is a vector space;
if $\varphi$ is generic and $G$ is generically $d$-rigid then this stress space has dimension $|E| - \rho(V, d)$.

Kalai \cite{Kalai-LBT} observed that for $d\ge 3$ the graph of a simplicial $d$-polytope $P$ is generically $d$-rigid, and used it to prove that the dimension of the stress space of this graph equals $g_2(\partial P)$. This provides an alternative proof of the lower bound theorem of Barnette \cite{Barnette:LBT-73}, where the minimizers are those $P$ with $g_2(\partial P)=0$. Kalai also showed that, for $d\ge 4$, $g_2(\partial P)=0$ if and only if
$P$ is \emph{stacked}, namely it can be obtained from the $d$-simplex by repeatedly stacking a $d$-simplex over a facet of the polytope already constructed. Further, for $d\ge 5$ this happens if and only if every vertex link is stacked.

\subsection{The Hausdorff metric}
For a point $x \in \R^d$ and $A\subset \R^d$ define $d(x, A) := \inf_{a\in A} |x-a|$ to be the distance from $x$ to $A$ in the usual Euclidean metric.
Let $A, B$ be two bounded subsets of $\mathbb{R}^d$. Define the Hausdorff distance between $A$ and $B$ by:
\begin{equation*}\dH(A,B):= \max\left\{ \sup_{a\in A}d(a,B), \, \, \sup_{b\in B} d(b,A)\right\}.\end{equation*}
It is easy to verify that $\dH$ defines a metric on the space of compact subsets of $\R^d$, and thus restricts to a metric on the space of convex bodies in $\R^d$.

\subsection{Approximation theory}
\label{subsec:approx}
Every convex body $K$ can be approximated by polytopes in the Hausdorff metric. A natural question is what is the minimal number of vertices that achieves an approximation of distance $\varepsilon$. Assume that $K$ is of type $C^1$. Let $n(\varepsilon)$ be the minimal number of vertices of a polytope $P$ with $\delta^H(P,K) < \varepsilon$. It is clear that $n(\varepsilon)$ goes to infinity as $\varepsilon$ goes to $0$.

 If $K$ is $C^2$ then the asymptotic behavior  of $n(\varepsilon)$ is well understood. B\"or\"oczky \cite[Theorem A(9)]{MR1770932} computed the asymptotic growth of $n(\varepsilon)$ explicitly, as follows:
\begin{thm}\label{thm:Boro} If $K$ is a $C^2$-convex body then: \begin{equation}\label{eqn:boro} \lim_{\varepsilon \to 0} n(\varepsilon)\varepsilon^{(d-1)/2} = 4^{\frac{1-d}2}\frac{\Theta_{d-1}}{V_{d-1}} \int_{\dK} \sqrt{\kappa} d(\dK),  \end{equation}
where $V_d$ is the volume of the unit $d$-ball, $\Theta_d$ is the covering density of $\mathbb{R}^d$ by unit $d$-balls, and $\kappa$ is the Gauss curvature.
\end{thm}

For our purposes, the important property of equation~(\ref{eqn:boro}) is that the right-hand side is strictly bigger than $0$ and bounded. In particular $n(\varepsilon)$ behaves roughly like $\varepsilon^{-\frac{d-1}2}$ for small enough $\varepsilon$.

\section{Warm up: rigidity and Kalai's conjecture for the unit $4$-ball}
\label{sec:g_2}
This section is devoted to proving Kalai's conjecture for simplicial $4$-dimensional polytopes approximating the unit $4$-ball, using rigidity theory. We then vastly generalize the ideas demonstrated here to prove the general case in the next section.

As mentioned in Subsection~\ref{subsec:approx}, Conjecture~\ref{Kalai}(i) holds for $k=1$ (for any $d$), so the first open case of this conjecture is $k=2,d=4$, and the most basic $C^1$-convex body to consider is the unit $4$-ball. For the rest of the section, the support of a stress $\omega$ of an embedded graph $G$ is the set of vertices that belong to an edge $e$ of $G$ such that $w(e) \not = 0$.

\begin{lemma}\label{lem:links2} Let $P$ be a generically embedded simplicial $4$-polytope and let $v$ be a vertex of $P$. Assume that $\lk(v)$ is not stacked. Then there is a non-zero stress $w$ supported in $N_2(v):=\{u\in V(P):\ d(u,v)\le 2\ \textrm{in the graph metric}\}$.
\end{lemma}
\begin{proof}
We follow ideas of Kalai~\cite{Kalai-LBT}.
Recall that a simplicial $3$-polytope is stacked if and only if its $1$-skeleton is \emph{chordal}, cf.\ \cite[Theorem 8.5]{Kalai-LBT}, namely, all its induced cycles have length $3$. As $\lk(v)$ is not stacked, there exists an induced cycle $C = v_1, \dots v_m$ of $\lk(v)$ with $m\ge 4$. There are two cases to consider:
\begin{enumerate}
\item[i.] $C$ is not induced in $\partial P$. Then for some $1\le i<j\le m$ there is an edge $e=\{v_i, v_j\}$ in $\partial P$ that is not in $\lk(v)$. By the Cone Lemma in rigidity, cf.\ \cite[Theorem 5]{Whiteley:cones}, as the graph ($1$-skeleton) of $\lk(v)$ is generically $3$-rigid, the graph $G$ of $v*\lk(v)$ is generically $4$-rigid. Thus, $G \cup\{e\}$ supports a nonzero stress $w$. The vertex support of $w$ is contained in the vertices of $G \cup\{e\}\subseteq N_1(v)$, thus also in $N_2(v)$.
\item[ii.] $C$ is induced in $\partial P$. Consider the complex $\Delta = \bigcup_{i=3}^m v_i*\lk(v_i)$. By the Gluing Lemma in rigidity, cf.\ \cite[Theorem 2]{Asi-Roth2}, the graph $G=\Delta^{(1)}$ is generically $4$-rigid (as all the cones are, and $v_i*\lk(v_i) \cap v_{i+1}*\lk(v_{i+1})$ contains a tetrahedron so this intersection has at least $4$ vertices). The edge $e=\{v_1, v_2\}$ is not an edge of $\Delta$, but both $v_1$, $v_2$ are vertices of $\Delta$. Thus the given embedding of $G \cup \{e\}$ has a nonzero stress. This stress is supported in $N_2(v)$ as desired.
\end{enumerate}
\end{proof}

\begin{thm}\label{thm:d=4}
Kalai's Conjecture~\ref{Kalai}(i) holds for the unit $4$-ball.
\end{thm}

\begin{proof}
Assume by contradiction that $g_2(P_n)\le g-1$ for all $n$, for some positive integer~$g$. Let $B$ denote the unit $4$-ball.

Then there exist $\varepsilon_1>0$ and $g$ oriented hyperplanes $H_1, \dots, H_{g}$ that intersect the interior of $B$ such that the corresponding negative sides intersected with $B$ are far from each other: $0 \not\in H_i^-$ for $1\le i \le g$ and for any $1\le i<j\le g$, $\varepsilon_1<\min\{d(x,y):\ x\in H^{-}_i \cap B,\  y\in H^{-}_j \cap B\}$.
Let $\delta=\min_{1\le i\le g} (1-d(0, H_i))$.

If a simplicial polytope $P$ well-approximates $B$, then all its edges must be short. Specifically, there exists $\frac{\delta}{2}>\varepsilon_2>0$ such that if $\delta^H(P,B)<\varepsilon_2$ then all edges of $P$ have length $<\min\{\frac{\delta}{2},\frac{\varepsilon_1}{4}\}$. (A quantitative estimate will be given in Section~\ref{sec:C^2}, when we compute effective lower bounds on the $g$-numbers for $C^2$-convex bodies.)

By rescaling and slightly moving the vertices, w.l.o.g. we may assume the approximating polytopes $P_n$ are generically embedded and contained in $B$.
Now, for $P\subseteq B$ as above, if in each cap $H^{-}_i \cap B$ there is a vertex $v_i$ of $\partial P$ whose link is not stacked, then by Lemma~\ref{lem:links2} there is a stress $w_i$ supported in $N_2(v_i)$. By the choice of $\varepsilon_2$, for all $1\le i<j\le g$,  $N_2(v_i)\cap N_2(v_j)=\emptyset$, and thus the $g$ stresses $w_i$ are linearly independent, yielding $g_2(P)\ge g$, a contradiction. It follows that there is $1\le i\le g$ for which all vertices $v$ of $P$ in $H^{-}_i \cap B$ have stacked links.
Denote $H=H_i$.

\begin{figure}[htb]
\begin{center}
\includegraphics[scale = 0.5]{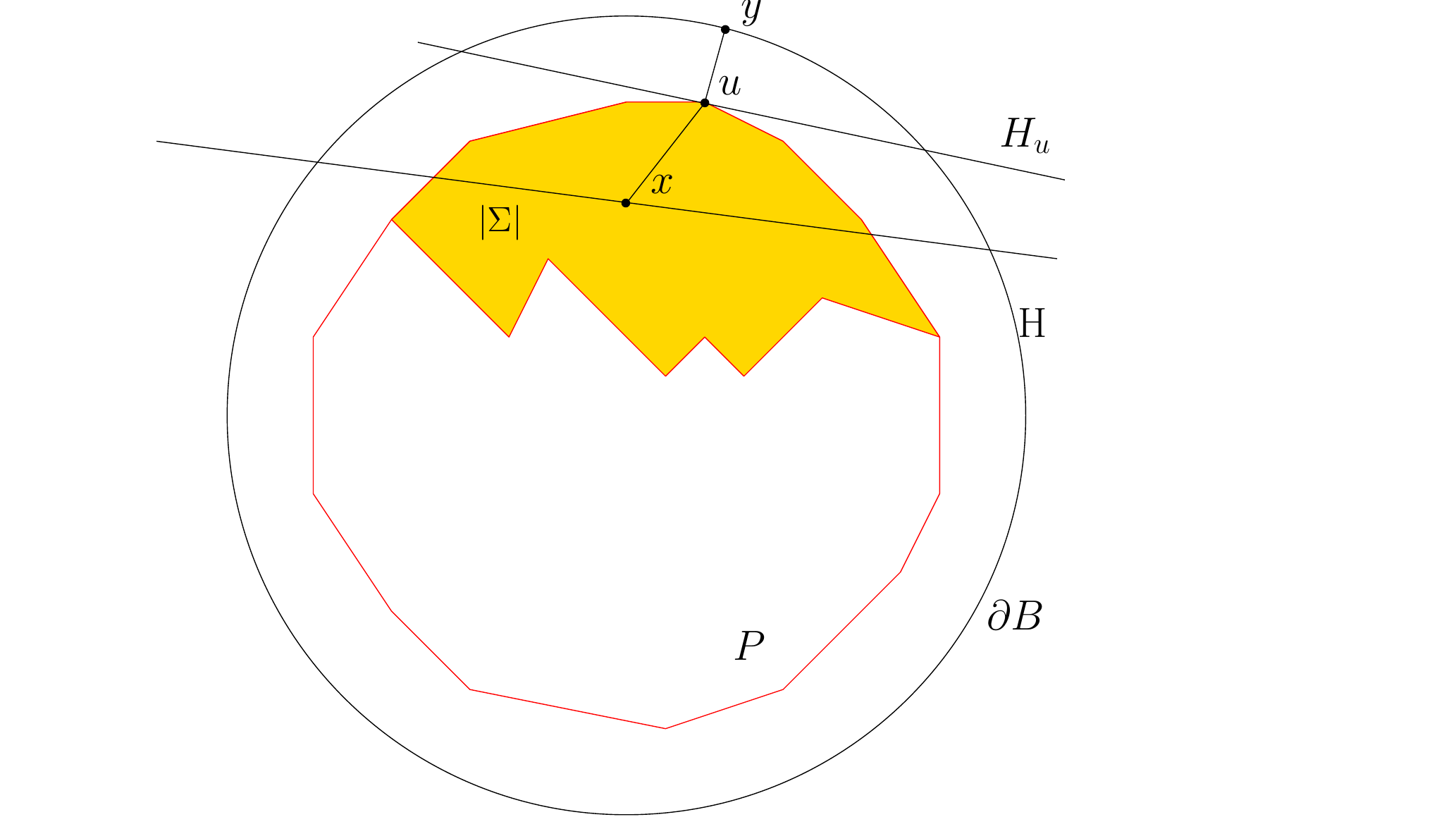}
\caption{$\delta^H(P,B) \ge d(y,u) \ge d(x, \partial B) - d(x,u) \ge \varepsilon_2$.}
\end{center}
\end{figure}

Note that for each such vertex $v$, $v*\lk(v)$ has a unique stacked triangulation $\Sigma_v$ (i.e., one without new vertices or edges). The family $\Sigma=\cup_{v\in V(P)\cap H^-}\Sigma_v$ is a geometric simplicial complex: indeed if $u$ and $v$ are vertices and $\Gamma$ is a simplex in $\Sigma_v$ that contains $u$ and  $v$, then the graph of $\Gamma$ is contained in $u*\lk(u)$ and by the uniqueness of the stacked triangulation $\Gamma \in \Sigma_u$ as well. In the next Lemma~\ref{lem:links} we will show that the geometric realization $|\Sigma|$ contains $H\cap P$ and thus has a point $x$ in the interior of some $4$-simplex $\sigma\in \Sigma$ with $d(x,\dB)\ge \delta$.

As all edges of $\sigma$ have length $<\frac{\delta}{2}$, all vertices of $\sigma$ are of distance $<\frac{\delta}{2}$ from $x$, and thus of distance $>\delta-\frac{\delta}{2}=\frac{\delta}{2}$ from $\dB$ by the triangle inequality. For a vertex $u\in\sigma$ and a supporting hyperplane $H_u$ of $P$ at $u$, the point $y$ in $\dB\cap H^-_u$ on the line orthogonal to $H_u$ through $u$ is of distance $>\frac{\delta}{2}>\varepsilon_2$ from $P$, a contradiction.
\end{proof}

\begin{lemma}\label{lem:links} Let $P$ be a simplicial $d$-polytope and let $H$ be a generic oriented hyperplane that passes through the interior of $P$. For each vertex $v$ of $P$ in $H^{-}$, let $\Sigma'_v$ be a triangulation of $\lk (v)$ and let $\Sigma_v$ be the collection of simplices formed by coning the simplices of $\Sigma'_v$ with $v$. Let $\Sigma \subset P$ be the family of all simplices of $\Sigma_v$ for all $v \in H^{-}$ and assume that it is a geometric simplicial complex. Then, for every point $x\in P\cap H$ there is a simplex $\Gamma \in \Sigma$ that contains $x$.
 \end{lemma}

 \begin{proof}
 Let $|\Sigma|$ be the set of points that belong to some simplex of $\Sigma$, so  $|\Sigma| \subseteq P$.  We need to show that $P\cap H \subseteq |\Sigma| \cap H$.  Since $H$ is generic it contains none of the nonempty faces of $P$ nor of $\Sigma$. Let $x \in P\cap H$ be generic, i.e. in general position, with respect to the vertices of $P\cap H$, and let $\ell$ be a generic line in $H$ through $x$. We claim that $|\Sigma|\cap \ell = P\cap \ell$. To establish this, note that $\ell\cap P$ is a closed line segment and admits a continuous parametrization $\gamma: [0,1] \to \ell\cap P$.

Assume that there is $x\in P\cap \ell$ that is not in $|\Sigma|$. Notice that $\gamma(0)$ lies in the relative interior of a facet of $\partial P$, so this facet contains a vertex $y \in H^{-}$, by genericity of $H$. This facet is contained in a $d$-simplex of $\Sigma_y$, thus $\gamma([0,z))$ is contained in $|\Sigma|$ for some positive real $z$. Let $s = \inf \{t\in [0,1] \, | \, \gamma(t) \not\in |\Sigma|\}$. Notice that $s\ge z > 0$. As $s>0$, by compactness of $|\Sigma|$ we conclude $\gamma(s) \in |\Sigma|$. By genericity of $\ell$, $\gamma(s)$ is in the relative interior of a $d$- or a $(d-1)$-simplex of $\Sigma$.

The former case is clearly not possible: $\gamma(s)$ would be in the interior of $|\Sigma|$ and therefore in the interior of $|\Sigma|\cap \ell$, a contradiction.
In the latter case we will show that $\gamma(s)$ is in the interior of $|\Sigma|$ unless it is in $\partial P$. The reason for this is the following: let $\Gamma$ be a $(d-1)$-simplex of $\Sigma$ that contains $\gamma(s)$. The ridge $\Gamma$ is contained in exactly two facets $F_1, F_2$ of the ball $\Sigma$ unless it is on the boundary of $P$; indeed, the boundary ridges of $\Sigma$ not on $\dP$ do not contain the vertex  $y \in H^{-}$ introduced in the preceding paragraph. If $\gamma(s)$ is not in $\partial P$ we obtain that $\gamma(s)$ is in the interior of $F_1 \cup F_2$, thus also in the interior of $|\Sigma|$. If $\gamma(s) \in \partial P$, then $s=0$ or $s=1$. The case $s=0$ was discarded before. The case $s=1$ says $P\cap l=|\Sigma|\cap l$.

It follows that $\ell\cap P \subseteq \ell\cap |\Sigma|$, so in particular $x\in |\Sigma|\cap H$. The set of generic points of $P\cap H$ is dense in $P\cap H$ and is contained in the closed set $|\Sigma|\cap H$. The desired inclusion follows.
 \end{proof}


\section{A proof of Kalai's conjecture for $C^1$-convex bodies}
\label{sec:C^1}
Here we prove the first main result of the paper, that part (i) of Kalai's conjecture is true. The following lemma is due to Zalgaller \cite{Zalgaller}, see also Schneider's book \cite[Section 2.3, Theorem in Note 1, case $s=1$]{Schneider:book2nd}.
\begin{lemma}\label{lem:auxHyp1} Let $K$ be a convex body in $\R^d$ and let $\pi$ denote an orthogonal projection onto a $k$-dimensional subspace $H$, chosen uniformly at random from the $(d,k)$-Grassmannian. Then, with probability $1$, all the affine subspaces that are orthogonal to $H$ and support $K$ do not contain a segment of $\dK$.
Thus, $\pi$ restricts to a homeomorphism from $K\cap \pi^{-1}(\partial \pi(K))$  to $\partial \pi(K)$.
\end{lemma}
\begin{figure}[htb]
\begin{center}
\includegraphics[scale= 0.5]{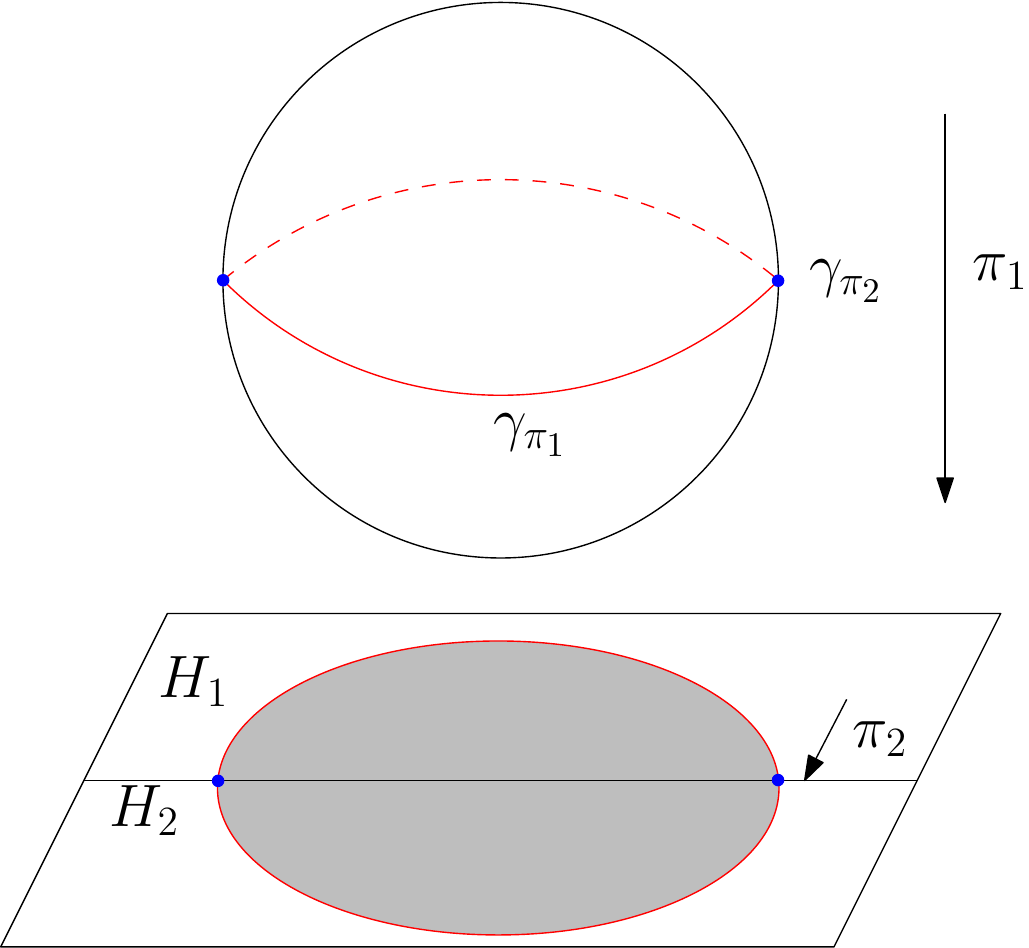}
\caption{Projections $\pi_1$ and $\pi_2$ to subspaces $H_1 \supseteq H_2$ with the respective $\gamma_{\pi_1} \cong \mathbb{S}^1$ and $\gamma_{\pi_2} \cong \mathbb{S}^0$.}
\end{center}
\end{figure}

In particular, the preimage of $\partial \pi(K)$ under $\pi$ is, with probability 1,  homeomorphic to a $(k-1)$-sphere; we denote this preimage by $\gamma_\pi$. Let $\gamma_\pi+\varepsilon := \gamma_\pi + \varepsilon \B$ (Minkowski sum), where $\B$ is the ball of radius $1$ in $\R^d$. Notice that in the Hausdorff metric

\begin{equation}\label{eqn:projInv}\lim_{\varepsilon\to 0} K\cap \pi^{-1}(\pi(\gamma_\pi+\varepsilon)) = \gamma_\pi.\end{equation}

Notice that if $\varepsilon$ is small enough, then there is a point $u\in\pi(K) \backslash \pi(\gamma_\pi + \varepsilon)$. Let $\hat{r}:\pi(\gamma_\pi + \varepsilon) \to \partial\pi(K)$ be the map that sends a point $x$ to the unique element $\hat{r}(x)$ of $\partial \pi(K)$ in the infinite ray from $u$ to $x$. The map $\hat{r}$ is a strong deformation retract if $\varepsilon$ is small enough. Now define $r: \gamma_\pi+ \varepsilon \to \gamma_\pi$ by letting $r(x)$ be the unique point in $\gamma_\pi$ that projects to $\hat{r}(\pi(x))$. Then $r$ is a strong deformation retract whenever $\hat{r}$ is, that is, for every small enough $\varepsilon$.

\begin{lemma}\label{lem:embededSurjects}
Let $\varepsilon>0$ be small enough so that the $\varepsilon$-neighborhood $\gamma_\pi+\varepsilon$  deformation retracts to $\gamma_\pi$.
Then, every simplicial polytope $P\subseteq K$ sufficiently close to $K$ in the Hausdorff metric,
has a subcomplex $\Delta\subseteq \dP \cap (\gamma_\pi+\varepsilon)$
whose embedding
into ${\gamma_\pi+\varepsilon}$ induces an isomorphism in homology.
\end{lemma}
\begin{proof}
Let  $\Delta :=P\cap \pi^{-1}(\partial(\pi(P)))$. Then $\Delta$ is a subcomplex of $P$ and $\pi(\Delta) = \partial (\pi(P))$. By equation (\ref{eqn:projInv}) there exists $\epsilon'>0$ such that $\pi^{-1}(\pi(\gamma_\pi +\epsilon')) \subseteq \gamma_{\pi} +\varepsilon$. If $P$ is close enough to $K$ then $\partial\pi(P) \subseteq \pi(\gamma_\pi + \epsilon')$, thus equation (\ref{eqn:projInv}) implies $\Delta$ is contained in $\gamma_{\pi}+\varepsilon$.
Note that $\pi_{| \Delta}$ is a homotopy equivalence from $|\Delta|$ to $\pi(\gamma_{\pi}+\varepsilon)$.

Let $g: \pi(\gamma_\pi) \to \gamma_\pi$ be the inverse of $\pi$ restricted to $\gamma_{\pi}$, namely $g(x)$ is the point $\pi^{-1}(x)\cap K$. Let $\iota$ denote the inclusion of $\Delta$ in $\gamma_\pi+\varepsilon$, then $r\circ \iota = g\circ \hat{r} \circ \pi_{| \Delta}$. The induced maps in homology of
$r,\  g,\ \hat{r},\ \pi_{| \Delta}$ are clearly isomorphisms, so $\iota$ is an isomorphism too.
 \end{proof}

Until now, we have not yet used the $C^1$ property of $K$ in any way.
Now we use the fact that all points of $\dK$ are non-singular. (In fact, this property is equivalent to being $C^1$.)

Consider any non-singular convex body $K$, let $(\varepsilon_i)$ denote a sequence of real positive numbers tending to $0$, and let $(P_i)$ denote a sequence of simplicial
polytopes so that $\dH(K,P_i)<\varepsilon_i$ for all $i$.

\begin{lemma}\label{lem:C^1used}
With $K$, $(\varepsilon_i)$ and $(P_i)$ as above, for every $\varepsilon>0$,
\[\max_{} \{\mathrm{diam}\, \sigma:{\sigma  \text{ is a face of } P_i,\ V(\sigma) \subset \gamma_\pi+\varepsilon_i,\ |\sigma| \not\subset \gamma_\pi+\varepsilon}\} \xrightarrow{i \rightarrow \infty} 0.\]
\end{lemma}
\begin{proof}
Assume by contradiction that there are $\delta>0$, a subsequence $(P_j)$ and
faces $\sigma_j\in P_j$ such that $V(\sigma_j) \subset \gamma_\pi+\varepsilon_j
,\ |\sigma_j| \not\subset \gamma_\pi+\varepsilon
$, and $\mathrm{diam}(\sigma_j)\ge \delta$.

There are two vertices in $\sigma_j$ whose distance is $\geq\delta$ and by the triangle inequality every point in $\sigma_j$ is at least $\frac \delta 2$ apart from one of them. Taking a point of $\sigma_j$ not in $\gamma_\pi +\varepsilon$  we obtain a line segment $e_j \subset \sigma_j$ of length at least $\frac{\delta}{2}>0$ connecting two points $v_j$,\ $v'_j$ such that $e_j \not\subset \gamma_\pi+\varepsilon $ and $v_j$ is a vertex of $\sigma_j$.

By compactness, passing to a subsequence we can assume that
there is convergence $e_j\xrightarrow{j \rightarrow \infty} e=[v,v']$ with $v\neq v'$
, and $v_j\xrightarrow{j \rightarrow \infty} v$, so $v\in\gamma_\pi$.

Since $ P_j  \xrightarrow{j \rightarrow \infty}  K$ then $e\subset \partial K$.
We claim that in fact $e$ must be contained in $\gamma_\pi$.
Notice that for any point $x$ in $\gamma_\pi$, the hyperplane $T_x$ tangent to $\partial K$ at $x$ projects to the tangent space to $\partial(\pi(K))$ at $\pi(x)$. As $T_v$ is the unique tangent plane at $v$, since $K$ is nonsingular, $e\subset T_v$ (and $e\subset T_{v'}$) and therefore $e\subset \gamma_\pi$ by Lemma~\ref{lem:auxHyp1}.

We conclude that for any fixed $\varepsilon>0$, a large enough $j$ satisfies $|e_j|\subseteq e +\varepsilon \subseteq \gamma_\pi +\varepsilon$, a contradiction to the choice of~$e_j$.
\end{proof}

%
%
%
%

Combining Lemmas~\ref{lem:embededSurjects} and~\ref{lem:C^1used} gives:
\begin{cor}\label{cor:C^1homology}
For any
non-singular
convex body $K$ and every $\varepsilon$ small enough, there is $\varepsilon'>0$ small enough such that for every simplicial polytope $P$ that is $\varepsilon'$-close to $K$ in the Hausdorff metric,  the subcomplex $\Gamma\subseteq\partial P$
induced by the vertices of $P$ in $\gamma_\pi+\varepsilon'$ is contained
in $\gamma_\pi+\varepsilon$, and this inclusion induces a surjection in homology.
\end{cor}
\begin{proof}
For small enough $\varepsilon>0$, $\gamma_\pi+\varepsilon$ retracts to $\gamma_\pi$. By Lemma~\ref{lem:C^1used}, there exists  $\varepsilon'<\frac{\varepsilon}{2}$ such that, if $P$ is $\varepsilon'$-close to $K$, for the complex $\Gamma$ on the vertices of $P$ in $\gamma_\pi +\varepsilon'$,
 all edges of $\Gamma$ that are not contained in $\gamma_\pi +\varepsilon'$ are of length $<\frac{\varepsilon}{2}$, so for a subcomplex $\Delta \subset \partial P \cap (\gamma_\pi+\varepsilon')$ as in Lemma~\ref{lem:embededSurjects} there are embeddings $|\Delta|\hookrightarrow |\Gamma|\hookrightarrow \gamma_\pi+\varepsilon$. Consider the induced maps in homology:
as the composition is a surjection in homology by Lemma~\ref{lem:embededSurjects}, so is the second map.
\end{proof}

We are now ready to prove Kalai's conjecture:
\begin{thm}\label{thm:main2}
Let $K$ be a $d$-dimensional $C^1$-convex body in $\mathbb{R}^d$ and let $g,k>0$ be integers with $k\le \frac d2$. There exists $\varepsilon > 0$ such that if $P$ is a simplicial polytope with $\delta^H(P, K) \le \varepsilon$, then $g_k(P) > g$.
\end{thm}

\begin{proof}
Let $x$ be an extremal point of $K$.
Then, by Lemma~\ref{lem:auxHyp1}, there are
\begin{compactitem}[$\triangleright$]
\item a projection $\pi$ of $\mathbb{R}^d$ onto a $(d-k)$-dimensional subspace, and
\item a ray $l$ emanating from $x$, $l\cap K=\{x\}$, such that the projection $\pi_l$ onto the orthogonal
space to $l$ contains the range of $\pi$, and
\item points $x\neq y_i \in l$ converging to $x$, and projective transformations $p_i$, each mapping $y_i$ to infinity along $l$,
\end{compactitem}
such that
\begin{compactitem}[$\triangleright$]
\item each composition $\pi_i=\pi\circ p_i$ restricts to a homeomorphism from $\gamma_{\pi_i}:=\pi_i^{-1}(\partial \pi_i(K))\cap K$ to $\partial \pi_i(K)$; so each $\gamma_{\pi_i}$ is a $(d-k-1)$-cycle, and
\item each $\pi_l \circ p_i$ restricts to a homeomorphism from $\gamma_{p_i}=(\pi_l \circ p_i)^{-1}(\partial \pi_l \circ p_i(K))\cap K$ to $\partial (\pi_l \circ p_i)(K)$; so each $\gamma_{p_i}$ is a $(d-2)$-cycle containing $\gamma_{\pi_i}$.
\end{compactitem}

\begin{figure}[htb]
\begin{center}
\includegraphics[scale = 0.5]{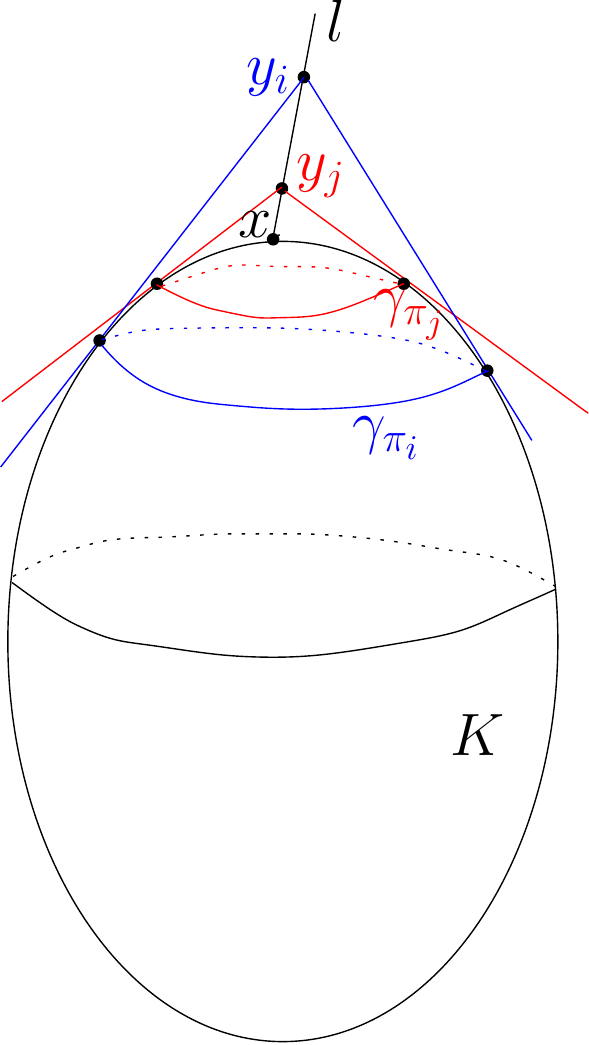}
\caption{Suitable values of $y_i$ give disjoint $\gamma_{\pi_i}$.}
\end{center}
\end{figure}


Note that $x\notin \gamma_{\pi_i}:=\gamma_i$ for all $i$, but $\gamma_{i}\to \{x\}$ in the Hausdorff measure.
By passing to a subsequence of $(y_i)$, we may assume that $\gamma_i \cap \gamma_j = \emptyset$ for all $i \neq j$: indeed, $x\not\in \gamma_i$ and for every $\varepsilon>0$ we have that $\gamma_j$ is contained in the open ball $B(x, \varepsilon)$ for sufficiently large $j$, so given $y_{i_j}$ we just need to pick $y_{i_{j+1}}$ so that $\gamma_{i_{j+1}} \subseteq B(x, d(x,\gamma_{i_j}))$.


Consider now the $(d-k-1)$-cycles $\gamma_1, \dots \gamma_{g+1}$. For $\varepsilon>0$ small enough, 
 the neighborhoods $\gamma_{i}+\varepsilon$ are pairwise disjoint and, for each $i$, $\gamma_{i}+\varepsilon$ deformation retracts to $\gamma_{i}$.
By Corollary~\ref{cor:C^1homology},
there is some $0<\varepsilon'<\varepsilon$ such that
the embedding of the induced complex $\Gamma_i$ on the vertices of $P$ in $\gamma_{i}+\varepsilon'$, into $\gamma_{i}+\varepsilon$, induces a surjection in homology.

It remains to show that for $\varepsilon$ small enough  and
for every $i\neq j$, there is no edge in $\dP$ between a vertex of $\Gamma_i$ and a vertex of $\Gamma_j$.  Once this is shown we get that
the complex $\Gamma=\cup_{1\le i\le g+1}\Gamma_i$ is an induced subcomplex of $\dP$, with $\tilde{\beta}_{d-k-1}(\Gamma) = \sum_{j=1}^{g+1} \tilde{\beta}_{d-k-1}(\Gamma_i) \ge g+1$, thus Theorem~\ref{thm:QLBT} finishes the proof.

Assume by contradiction there are approximating polytopes $(P_n)$ with $v_i(n)\in \Gamma_i(n)\subseteq P_n$,
$v_j(n)\in \Gamma_j(n)\subseteq P_n$, and $v_i(n)v_j(n)$ an edge of $P_n$.
Then there exist a subsequence $(P_{a_n})$ of $(P_n)$, a point $v_i\in \gamma_i$  with $v_i(a_n)\rightarrow v_i$ and a point $v_j\in \gamma_j$ with $v_j(a_n)\rightarrow v_j$. The segment $[v_i,v_j]$ is contained in $\dK$.
As $K$ is $C^1$, that is nonsingular, $[v_i,v_j]$ is contained in the unique hyperplane $H$ through $v_j$ that supports $K$. Thus,
 by the choice of $\pi_j$, also $v_i\in \gamma_{j}$, contradicting that $\gamma_{j}$ and $\gamma_{i}$ are disjoint.
\end{proof}

Let us remark that since the homology cycles $\gamma_i$ are represented by spheres, we can substitute, for self-containedness, the use of Theorem~\ref{thm:QLBT} in the proof of Theorem~\ref{thm:main2} by the following simple lemma.

Let $\Delta$ be a simplicial complex with a map $\varphi$ of its vertex set $V$ to $\R^d$, and let \[\widetilde{\varphi}:V\longrightarrow \R^d\times \{1\}\ \subset\ \R^{d+1}\] the homogenization of $\varphi$. An \emph{(affine) $k$-stress} is a map $\omega$ from the $(k-1)$-dimensional faces of $\Delta$ to $\R$ such that, for every $(k-2)$-face $\tau$ of $\Delta$, the \emph{Minkowski balancing condition} is satisfied, i.e.
\begin{equation*}
\sum_{\substack{\sigma:\ \sigma\ (k-1)\text{-face},\ \tau \subset \sigma}} \omega(\sigma)(\widetilde{\varphi}(\sigma\setminus \tau)) = 0\ \mod \mathrm{span}(\widetilde{\varphi}(\tau)),
\end{equation*}
namely, \ the sum on the left-hand side lies in the linear span of $\widetilde{\varphi}(\tau)$. We refer to Lee \cite{MR1384883} for a comprehensive introduction to the subject of affine stresses.

As such, a stress on a graph is the same as a $2$-stress (we will henceforth leave out the quantifier ''affine``). Moreover, it
turns out
that $k$-stresses are special $(k-1)$-cycles in the simplicial chain complex of $\Delta$ with real coefficients, see Ishida \cite{MR951199} and Tay--Whiteley \cite{MR1773196} for the associated homology theories and more background on the notions. We turn back to the problem at hand.

\begin{lemma}\label{lem:QLBT}
Let $\gamma$ denote a simplicial $(k-1)$-sphere on vertex set $W$. Assume $\gamma$ is realized as a subcomplex of the boundary $\Delta$ of a simplicial $d$-polytope $P$, where $k\le \frac{d}{2}$. Assume furthermore that the fundamental class of $\gamma$ defines a nontrivial homology class in $\tilde{H}_{k-1}(\Delta_W)$.
Then the simplicial neighborhood  \[\Gamma\ :=\ \{\sigma \in \Delta: \exists \tau \in \Delta,\ \sigma \subset \tau,\ \tau \cap \gamma\neq \emptyset\}\] of $\gamma$ in $\Delta$ supports a $k$-stress homologous to the fundamental class of $\gamma$ in $\Gamma$.
\end{lemma}

We refer the reader to the appendix for a proof of the lemma.
To finish the alternative proof of Theorem~\ref{thm:main2}, it suffices to recall the central corollary from the hard Lefschetz theorem for projective toric varieties together with the fact that stresses supported on disjoint vertex sets are linearly independent:

\begin{thm}{\rm{\cite[Theorems 6.1 \& 7.3 and p.431]{MR1228132}}}
For any simplicial $d$-polytope $P$, and $k\le \frac{d}{2}$,
\[g_k(P)\ =\ \dim \{\text{space of $k$-stresses supported in $P$}\}.\]
\end{thm}


\section{Refined bounds for $C^2$-convex bodies}
\label{sec:C^2}
In the case that $K$ is of type $C^2$ the asymptotic growth of $g_k$ can be bounded below. We start by computing these bounds for approximations of the unit ball and then use tricks of B\"or\"oczky to pass to the case of general $C^2$-convex body. The idea is to use the quantitative lower bound Theorem~\ref{thm:QLBT} (or Lemma~\ref{lem:QLBT}) again and to provide such bounds by finding many cycles in $\partial K$ that are disjoint and far from each other.

Let $\B$ be the unit ball in $\R^d$. The following lemma is known:
\begin{lemma}\label{lem:sphCod}
For every sufficiently small $\varepsilon$ there is a subset $A$ of the boundary of $\B$ with $|A|=\Omega(\varepsilon^{1-d})$ and distance $d(x,y) \ge \varepsilon$ for every pair of points $x,y\in A$.
\end{lemma}
\begin{proof} Pick an orthogonal basis of $\R^{d-1}\times\{0\}$ with vectors of length $\varepsilon$. Consider the intersection of the lattice generated by this basis and $\B$ and lift it to the boundary of $\B$ to obtain the set $A$. That $A$ works.
\begin{figure}[htb]
\begin{center}
\includegraphics[scale = 0.4]{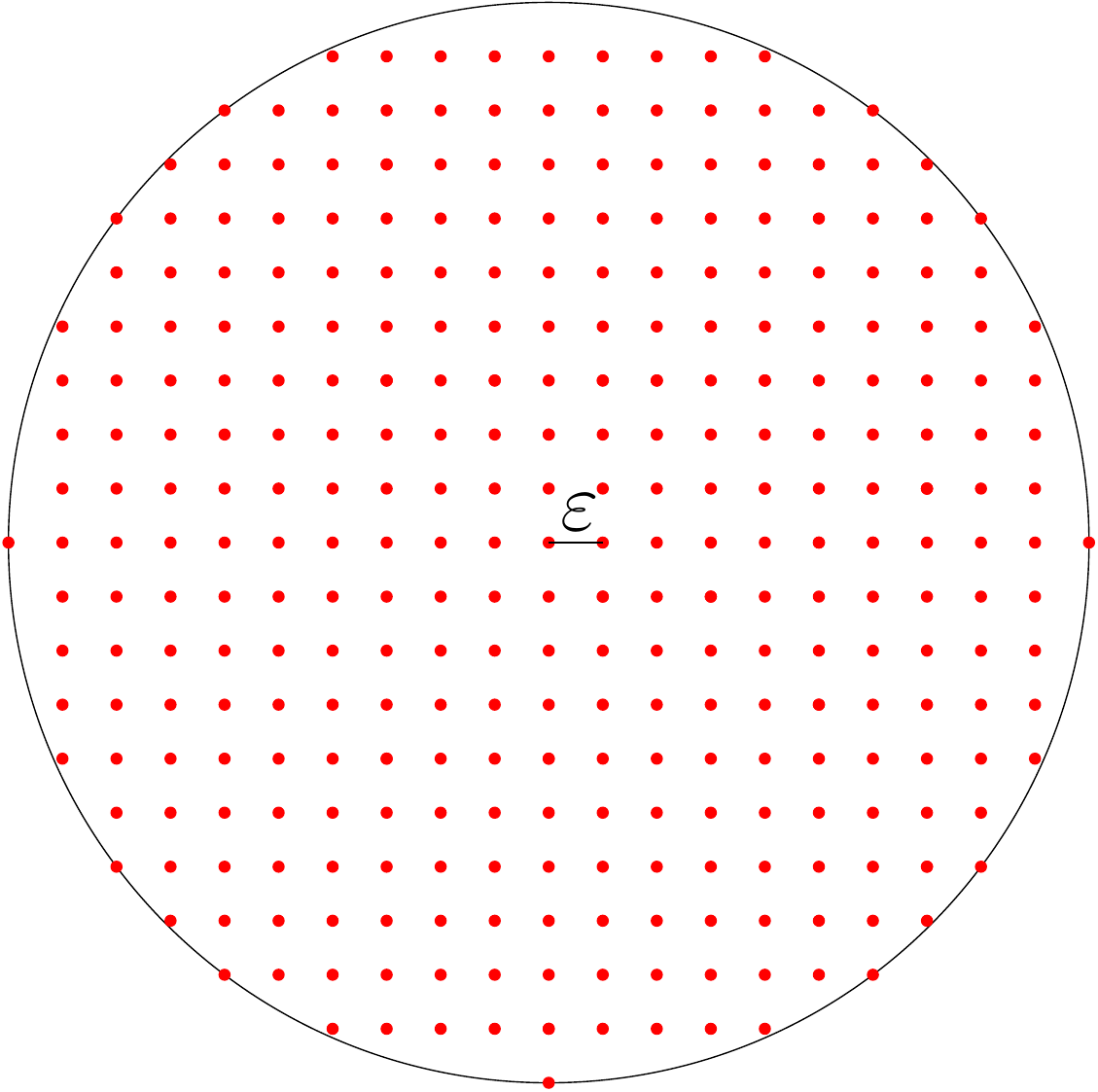}
\caption{The set $A$ for $\mathbb{S}^2$ projected to $\mathbb{R}^2$.}
\end{center}
\end{figure}
\end{proof}

We are now in a position to provide quantitative lower bounds for the $g$-numbers when approximating the unit ball.

\begin{thm}\label{thm:QSLBT}
Let $k\le \frac d2$.
If $\dH(P,\B)$ is small enough
then
\begin{equation}\label{eqn:QSLBT}
g_k(P) = \Omega\left(\delta^H(P, \B)^{\frac{1-d}2}\right).
\end{equation}
\end{thm}
\begin{proof}
The idea is to intersect $P$ with $(d-k)$-dimensional affine subspaces, where each subspace is close to a different point of a set $A$ from the previous lemma. Then the induced complex on vertices of $P$ that are close enough to these intersections will have $\tilde{\beta}_{d-k-1} \ge |A|$ (with contribution of at least one $(d-k-1)$-cycle per intersection). Here are the details.

Let $\varepsilon > 0$ be sufficiently small so by Lemma~\ref{lem:sphCod} there is a set $A$ of points of the boundary of $\B$ with cardinality $\Omega\left(\varepsilon^{\frac{1-d}2}\right)$ such that the $d(x,y) >35\varepsilon^{\frac 12}$ for every $x, y \in A$.

For each $x \in A$  let $H_x$ be the affine hyperplane `below $x$' such that $d(x,y) = 11\varepsilon^{\frac 12}$ for every $y\in  \partial \B\cap H_x$, and let $L_x$ be any $(d-k)$-dimensional subspace contained in $H_x$ that passes through the center $u(x)$ of the ball $H_x\cap \B$.


Let $P$ be a simplicial polytope with $\delta^H(P,\B) < \varepsilon$ and boundary complex $\Delta=\dP$. By rescaling $P$ (multiplying by $(1+\varepsilon)^{-1}$) we obtain a polytope contained in $\B$, combinatorially equivalent to $P$ and whose distance to $\B$ is smaller than $2\varepsilon$, so it is enough to assume that $P\subseteq \B$. If $\varepsilon$ is small enough, then the length of an edge $e\in \partial P$ is bounded above by $4\varepsilon^{\frac 12}$. To see this, apply the Pythagorean theorem to the triangle in the plane spanned by $e$ and the origin, whose vertices are the origin, the intersection point of the line spanned by $e$ and the line orthogonal to it through the origin, and the appropriate end point of $e$.

For each $x\in A$ let $W_x$ be the set of all vertices of $P$ contained in a face that intersects $L_x$. Then for any vertex $v\in W_x$, $d(v,L_x)\le 4\varepsilon^{\frac 12}$, as it is bounded by the length of the longest edge of a face containing $v$ that intersects $L_x$.



Let $\Delta_W$ be the complex induced by the vertices in $W:=\bigcup_{x\in A} W_x$. For points $x\neq y$ in $A$, and vertices $v\in W_x,\ u\in W_y$,
the triangle inequality yields $|v-u| \ge \varepsilon^{\frac 12}(35-11-4-11-4) = 5\varepsilon^{\frac 12}$. As the longest edge in $P$ has length $\le 4\varepsilon^{\frac 12}$, we conclude that $\Delta_W$ is the disjoint union of the subcomplexes $\Delta_{W_x}$, for all $x\in A$.

\begin{figure}[htb]
\begin{center}
\includegraphics[scale = 0.5]{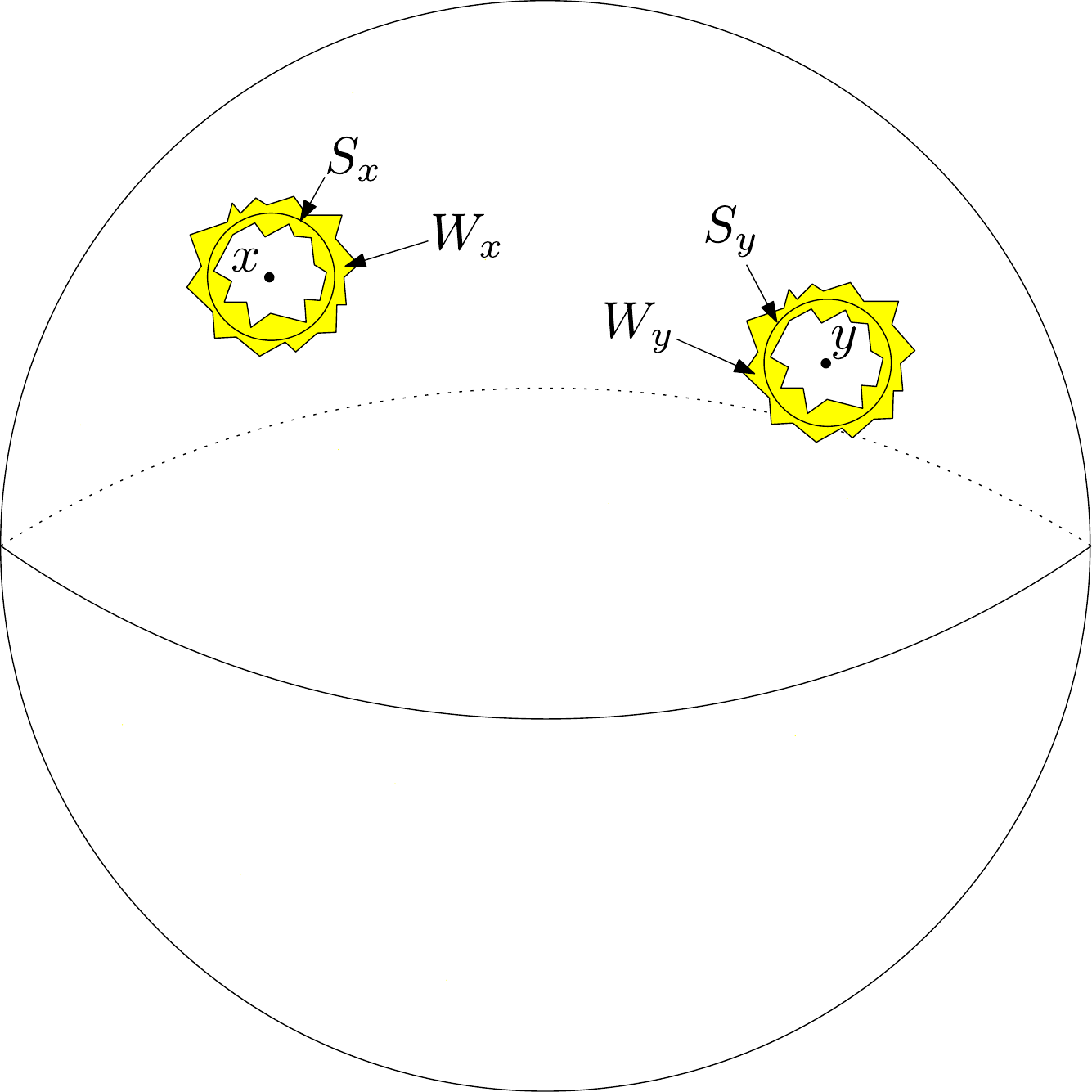}
\caption{Any $W_x$ and $W_y$ are disjoint and far from each other.}
\end{center}
\end{figure}
We claim that, for $\varepsilon>0$ small enough,
$\beta_{d-k-1}(\Delta_{W_x}) \ge 1$.
The argument is similar to, and simpler than, the one we used in the proof of Theorem~\ref{thm:main2}:
let $S_x:=L_x\cap \partial \B$. Then clearly for small enough $\varepsilon>0$ there exists $\varepsilon'>0$ such that $S_x+\varepsilon'$ contains the strip $\partial\B \cap (L_x+4\varepsilon^{\frac12})$ and is homotopy equivalent to $S_x$.
Then the composition of the following maps induces an isomorphism in homology: \[ \partial P \cap L_x \hookrightarrow \Delta_{W_x} \hookrightarrow S_x+\varepsilon',\]
where both ends are nontrivial singular $(d-k-1)$-cycles.
Thus $\tilde{\beta}_{d-k-1}(\Delta_{W_x})\ge 1$.

It then follows from Theorem~\ref{thm:QLBT} that
\begin{equation} \label{eqn:Q}
g_k(P) \ge \beta_{d-k-1}(\Delta_W) = \sum_{x\in A} \beta_{d-k-1}(\Delta_{W_x}) \ge |A| = \Omega\left(\varepsilon^{\frac{1-d}2}\right).
\end{equation}
\end{proof}
In fact, instead of using Theorem~\ref{thm:QLBT}, it suffices to use Lemma~\ref{lem:QLBT}, as taking a $(k-1)$-sphere in each of
 the $S_x$'s gives $|A|$ pairwise disjoint $(k-1)$-spheres and thus they correspond to linearly independent $k$-stresses.

\begin{cor}\label{lem}
Let $E$ be an ellipsoid and let $k\le \frac d2$.
If $\dH(P,E)$ is small enough
then
\begin{equation}\label{eqn:QSLBT2}
g_k(P) = \Omega\left(\delta^H(P, E)^{\frac{1-d}2}\right).
\end{equation}
\end{cor}
\begin{proof}
There is an affine transformation that maps $E$ to $\B$. Affine transformations map any polytope to a combinatorially equivalent polytope and the distances are preserved up to a constant, so the result follows from Theorem~\ref{thm:QSLBT}.
\end{proof}

\begin{thm}\label{thm:main3}
Let $K$ be a $C^2$-convex body and let $k\le \frac d2$.
If $\dH(P,K)$ is small enough
then
\begin{equation}\label{eqn:QSLBT3}
g_k(P) = \Omega\left(\delta^H(P, K)^{\frac{1-d}2}\right).
\end{equation}
\end{thm}

\begin{proof}
Let $x$ be a point in $\partial K$ of positive curvature and let $E$ be the tangent conic to $\partial K$ at $x$ given by the Hessian of $\partial K$ at $x$. Then there is a neighborhood of $x$ in $\partial K$ that lies between $(1+\varepsilon) E-\varepsilon x$ and $(1-\varepsilon) E+\varepsilon x$. This follows from the fact that $K$ and $E$ have the same tangent space at $x$ and the same Hessian, thus the error in approximation is of the third order (see Schneider \cite[Chapter 2.5]{Schneider:book2nd}. The projections to the tangent plane at $x$ gives a homeomorphism between neighborhoods of $x$ in $\partial K$ and $\partial E$  that allows to transfer cycles in $\partial E$ of Lemma~\ref{lem} to cycles in $\partial K$. Approximating those cycles give the desired lower bound as in Theorem~\ref{thm:QSLBT} since the number of cycles in ellipsoids can be estimated locally up to a constant.
\end{proof}

\begin{rem}
Using arguments of B{\"o}r{\"o}czky \cite{MR1742207}, one can refine Theorem~\ref{thm:main3} to show that, for some constant $C$ independent of $K$, $k$ and $d$, we have
\[g_k(P)\ge C \cdot \left(\int_{\partial K} \kappa(x)^{\frac{1}{d+1}}\right)^{-\frac{d+1}{d-1}} \cdot \left(d\cdot \delta^H(P, K)\right)^{\frac{1-d}2}\]
where $\kappa(x)$ is the determinant of the second fundamental form.
\end{rem}
\begin{rem}[Tightness]
Notice that, by B\"or\"oczky's  \cite[Theorem B]{MR1742207}, in an optimal $\varepsilon$-approximation of the unit $d$-ball $\B$ by a polytope $P$, the number $f_k$ of $k$-faces is bounded above by $C\varepsilon^{\frac{1-d}2}$, where $C$ is a constant depending only on $\B$. Since $g_k\le f_{k-1}$, the lower bound is tight up to a constant. Again, this result holds for other distance notions as well, and extends to approximations of convex bodies with $C^2_+$ boundary, i.e convex bodies whose Gaussian curvature is positive at every boundary point.
\end{rem}
\begin{rem}[Upper bounds: Conjecture~\ref{Kalai}(ii) holds for random polytopes.]\label{rem:barany}

B\'{a}r\'{a}ny \cite[Theorem 6, Corollary 2]{Barany} showed
that if $P_n$ is a polytope obtained from sampling $n$ points uniformly at random in
a $C^2$-convex body $K$ then $\mathbb{E}(\delta^H(P_n, K))= \Theta(\left(\frac{\log n}{n}\right)^{\frac{2}{d+1}})$. (B\'{a}r\'{a}ny assumed positive Gaussian curvature, but B\"or\"oczky's results show this assumption is not needed.)
Furthermore, for any $0\le k \le d-1$ he showed that $\mathbb{E}(f_k(P_n)) = 
\Theta(n^{\frac{d-1}{d+1}})$.

Combining this with Theorem~\ref{thm:main3} we conclude that part (ii) of Kalai's conjecture holds for \emph{random} simplicial polytopes.
Indeed, for $\frac{1}{m+1}\leq \dH(P,K)<\frac{1}{m}$ small enough, $g_k(P)$ is of order $m^{\frac{d-1}{2}+o(1)}$ for all $1\le k\le \frac{d}{2}$, thus $\partial^k(g_k(P))=O(m^{\frac{(d-1)(k-1)}{2k}+o(1)})=o(g_{k-1}(P))$, and (ii) follows from (i).
\end{rem}

\appendix \enlargethispage{3mm}
\section{Appendix: From induced homology cycles to affine stresses}

The purpose of this section is to prove Lemma~\ref{lem:QLBT}. Let us first observe a simpler and at first insufficient lemma that gets us almost to the goal.

\begin{lemma}\label{lem:QLBTi}
Let $\gamma$ denote a simplicial $(k-1)$-sphere contained in the boundary of a simplicial $d$-polytope $P$, where $k\le \frac{d}{2}$. Assume that $\gamma$ is an induced subcomplex in $\Delta=\partial P$.
Then the simplicial neighborhood $\Gamma$ of $\gamma$ in $\Delta$ supports a $k$-stress homologous to the fundamental class of $\gamma$ as a cycle in $\Gamma$.
\end{lemma}

\begin{proof}
The simplicial neighborhoods of vertices $v$ in $\Delta$ are denoted by $\St_\Delta v$, and their interiors are denoted by $\St_\Delta^\circ v$.

The proof of Lemma \ref{lem:QLBTi} is now the same as for induced cycles in the graph of $P$ by Kalai \cite{Kalai-LBT} that we used in Lemma~\ref{lem:links2}:
Consider the family \[(U_v:=\St_\Delta^\circ v)_{v \text{ vertex of $\gamma$}}.\]
This is a good cover of the open set $\cup_{v \text{ in $\gamma$}}U_v$.
The generalized Mayer--Vietoris principle given by the \v{C}ech complex of this cover gives a double complex whose spectral sequences compute homology groups of the nerve $\mathcal{N}$ of~$(U_v)$. Instead of applying this to compute the usual homology groups, however, we can also apply this to Ishida's chain complex \cite{MR951199} for stress groups, which is worked out in detail by Tay--Whiteley \cite[Sections 10 \& 12]{MR1773196}.

With this, we get straightforwardly a natural surjection
\begin{equation}\label{eq:surj}
\left\{\text{$i$-stresses of $\bigcup_{v \text{ vertex of $\gamma$}} \St_\Delta v$ }\right\}\ \longtwoheadrightarrow\ \widetilde{H}_{i-1}(\mathcal{N})
\end{equation}
where $i$ is the smallest integer $\le \frac{d}{2}$ so that $\widetilde{H}_{i-1}(\mathcal{N})$ is nontrivial.
But $\mathcal{N}$ is just the simplicial sphere $\gamma$ itself, because $\gamma$ is induced in $\Delta$. The claim follows.
\end{proof}

\begin{rem}[An even more elementary approach]
Following Provan and Billera \cite{MR593648}, a simplicial complex is \emph{vertex-decomposable} if it consists of a single facet, or it is pure and there is some vertex of the simplicial complex whose link and deletion are both vertex-decomposable. Here, the \emph{deletion} of a vertex from a simplicial complex is the subcomplex induced by the facets not containing that vertex.

If, in the situation of Lemma~\ref{lem:QLBTi}, $\gamma$ is {vertex-decomposable}, then we can use that fact to replace the use of the generalized Mayer--Vietoris principle by a simple argument relying on the usual Mayer--Vietoris exact sequence.
To see this, iteratively apply the Mayer--Vietoris sequence to the Ishida complex and the simplicial chain complex in parallel when building $\bigcup_{v \text{ vertex of $\gamma$}} \St_\Delta v$ vertex after vertex along the vertex-decomposition of $\gamma$.
Comparing both exact sequences obtained, which are connected by the natural map sending chains of the Ishida complex to simplicial chains, we recover the surjection \eqref{eq:surj} in this restricted setting by a straightforward induction.

To justify this detour, we remark that we are solely interested in applying Lemma~\ref{lem:QLBTi} when $\gamma$ is combinatorially isomorphic to the barycentric subdivision of a polytope boundary, and therefore is indeed vertex-decomposable 
~\cite[Corollary~3.3.3]{MR593648}. Hence, this simpler (but more technical, and perhaps less transparent) reasoning would also fully suffice for our purposes.
\end{rem}

We now are left with the delicate task of extending this to a proof of Lemma~\ref{lem:QLBT}.

\begin{proof}[\textbf{Proof of Lemma~\ref{lem:QLBT}}]
Consider $\Delta=\partial P$ and $\Delta'$ the result of a barycentric subdivision of all faces intersecting $\gamma$, see Figure~\ref{fig:cycle}. Then the corresponding subdivision $\gamma'$ of $\gamma$ in $\Delta'$ is an induced subcomplex of $\Delta'$, and of the induced subdivision of $\Gamma$, denoted $\Gamma'$. We consider also the open simplicial complex
\[\widecheck{\Gamma}\ :=\ \{\sigma \in \Gamma: \sigma \cap \gamma\neq \emptyset\}\]
and its induced subdivision in $\Gamma'$, denoted by $\widecheck{\Gamma}'$.

\begin{figure}[htb]
\begin{center}
\includegraphics[scale = 0.5]{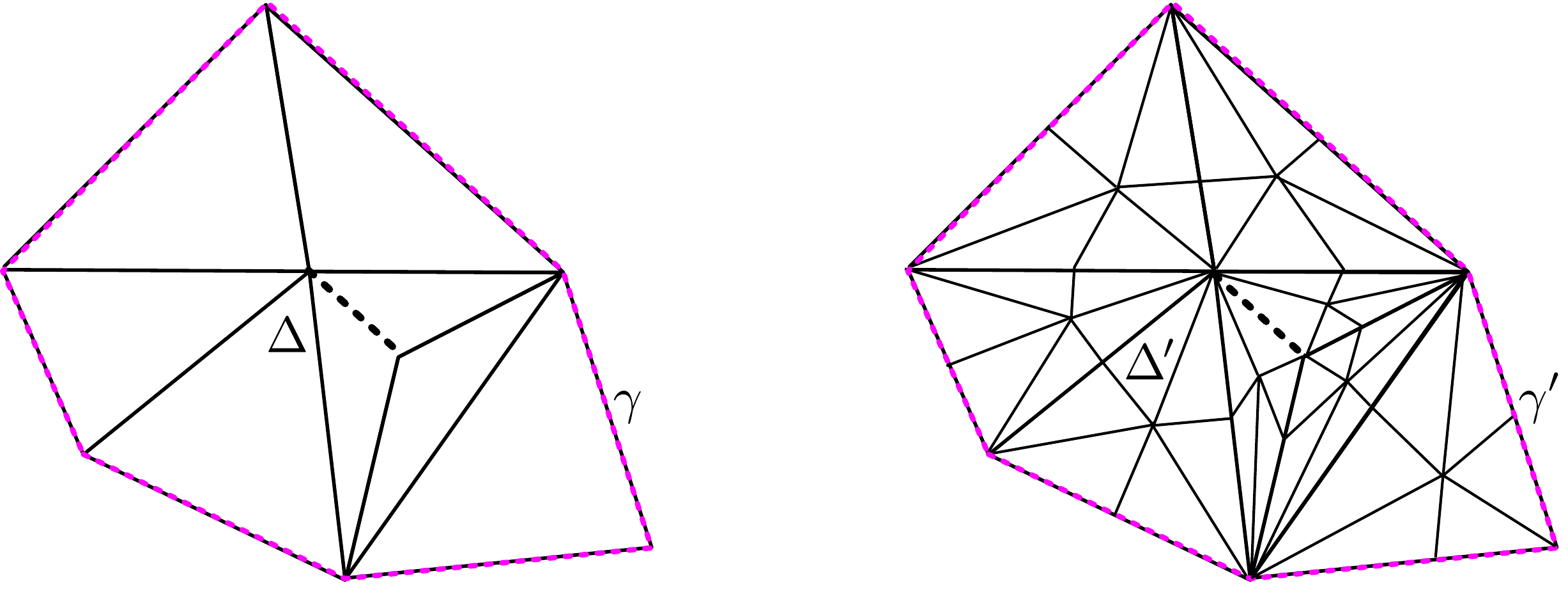}
\caption{Barycentric subdivision of a simplicial complex $\Delta$ at faces intersecting a cycle~$\gamma$. In the resulting complex $\Delta'$, the cycle $\gamma'$ subdividing $\gamma$ is an induced subcomplex. The dotted black edge is not subdivided in the process, as it is not incident to $\gamma$.}
\label{fig:cycle}
\end{center}
\end{figure}

\noindent The barycentric subdivision is algebraically realized by iterative blowups of the toric variety, or, combinatorially, by stellar subdivisions of $\Delta$, performed at faces intersecting $\gamma$ in order of decreasing dimension.
Following McMullen \cite{MR1228132}, this preserves the validity of the hard Lefschetz theorem. Hence we can apply Lemma~\ref{lem:QLBTi} to conclude that the fundamental class of $\gamma'$ is naturally homologous to a stress $\widetilde{\gamma}$ in its simplicial neighborhood $\widetilde{\Gamma}\subset \Gamma'$.

Now, we have to be careful since a stellar subdivision in the blowup sequence may introduce stresses on its own.
To control this, consider a simplicial $d$-polytope $X$ and its stellar subdivision $X'$ at a face $\sigma$.
Following McMullen again, the stresses of $X'$ are decomposed into pullbacks of stresses of $X$, and Gysin pullbacks of stresses in the face figure $X_\sigma$ of $\sigma$ in $X$, the latter of which are the ``new'' stresses to be controlled, see also
\cite[Thm.1.2(3)]{Babson-Nevo},
 \cite[Section 6]{MR1644323}, and \cite[Theorem 6.18]{AHP} for a detailed presentation of the Gysin maps involved.
It is straightforward to see that the latter stresses are naturally supported in the simplicial neighborhood of the link $\lk_\Delta \sigma$ of $\sigma$.

We conclude that all newly created $k$-stresses in the transition from $\Delta$ to $\Delta'$, where $k\ge 2$, seen as $(k-1)$-cycles, are supported in $\Gamma'$. Moreover, they are zero-homologous as simplicial cycles in $\widecheck{\Gamma}'$.

But the stress $\widetilde{\gamma}$ generates a nontrivial homology class in $\widecheck{\Gamma}'$, so it is linearly independent of the newly created stresses. Hence, we may blow down again, which maps $\widetilde{\gamma}$ to a nontrivial stress supported in the simplicial neighborhood of~$\gamma$, as desired.
\end{proof}

{\small
\providecommand{\bysame}{\leavevmode\hbox to3em{\hrulefill}\thinspace}
\providecommand{\MR}{\relax\ifhmode\unskip\space\fi MR }
\providecommand{\MRhref}[2]{%
  \href{http://www.ams.org/mathscinet-getitem?mr=#1}{#2}
}
\providecommand{\href}[2]{#2}

}
\end{document}